\documentclass{article}

\usepackage{cmbright}
\usepackage{amssymb,amsmath,amsthm}
\usepackage{mathrsfs}
\usepackage{epsfig}
\usepackage[usenames,dvipsnames]{color}

\def\B{\mathscr B}
\def\C{\mathbb C}
\def\d{\mathrm{d}}

\def\H{\mathcal H}

\def\N{\mathbb N}

\def\R{\mathbb R}
\def\S{\mathbb S}
\def\T{\mathbb T}

\def\Z{\mathbb Z}

\def\dom{\mathcal D}

\def\e{\mathop{\mathrm{e}}\nolimits}

\def\linf{\mathsf{L}^{\:\!\!\infty}}
\def\ltwo{\mathsf{L}^{\:\!\!2}}

\DeclareMathOperator*{\slim}{s\;\!-lim\;\!}

\newtheorem{Theorem}{Theorem}[section]

\newtheorem{Lemma}[Theorem]{Lemma}
\newtheorem{Assumption}[Theorem]{Assumption}

\begin{document}

%--------------------------------------------------------------------------------------
% Title
%--------------------------------------------------------------------------------------

\title{Furstenberg transformations on Cartesian products of infinite-dimensional tori}

\author{P. A. Cecchi$^1$ and R. Tiedra de Aldecoa$^2$\footnote{Supported by the Chilean
Fondecyt Grant 1130168 and by the Iniciativa Cientifica Milenio ICM RC120002
``Mathematical Physics'' from the Chilean Ministry of Economy.}}

\date{\small}
\maketitle \vspace{-1cm}

\begin{quote}
\emph{
\begin{itemize}
\item[$^1$] Departamento de Matem\'aticas, Universidad de Santiago de Chile,\\
Av. Alameda Libertador Bernardo O'Higgins 3363, Estaci\'on Central, Santiago, Chile
\item[$^2$] Facultad de Matem\'aticas, Pontificia Universidad Cat\'olica de Chile,\\
Av. Vicu\~na Mackenna 4860, Santiago, Chile
\item[] \emph{E-mails:} pcecchib@gmail.com, rtiedra@mat.puc.cl
\end{itemize}
}
\end{quote}

%--------------------------------------------------------------------------------------

\begin{abstract}
We consider in this note Furstenberg transformations on Cartesian products of
infinite-dimensional tori. Under some appropriate assumptions, we show that these
transformations are uniquely ergodic with respect to the Haar measure and have
countable Lebesgue spectrum in a suitable subspace. These results generalise to
the infinite-dimensional setting previous results of H. Furstenberg, A. Iwanik,
M. Lema\'nzyk, D. Rudolph and the second author in the one-dimensional setting.
Our proofs rely on the use of commutator methods for unitary operators and Bruhat
functions on the infinite-dimensional torus.
\end{abstract}


\smallskip


%--------------------------------------------------------------------------------------
\section{Introduction}
\setcounter{equation}{0}
%--------------------------------------------------------------------------------------

We consider in this note the generalisation of Furstenberg transformations on
Cartesian products $\T^d$ ($d\ge2$) of one-dimensional tori \cite[Sec.~2]{Fur61} to
the case of Cartesian products $(\T^\infty)^d$ of infinite-dimensional tori. Using
recent results on commutator methods for unitary operators \cite{Tie14_2,Tie13}, we
show under a $C^1$ regularity assumption on the perturbations that these
transformations are uniquely ergodic with respect to the Haar measure on
$(\T^\infty)^d$ and are strongly mixing in the orthocomplement of functions depending
only on variables in the first torus $\T^\infty$ (see Assumption \ref{ass_phi} and
Theorem \ref{thm_erg}). Under a slightly stronger regularity assumption ($C^1$ + Dini
continuous derivative), we also show that these transformations have countable
Lebesgue spectrum in that orthocomplement (see Theorem \ref{thm_Lebesgue}). These
results generalise to the infinite-dimensional setting previous results of H.
Furstenberg, A. Iwanik, M. Lema\'nzyk, D. Rudolph and the second author
\cite{Fur61,ILR93,{Tie13}} in the one-dimensional setting.

Apart from commutator methods, our proofs rely on the use of Bruhat test functions
\cite{Bru61} which provide a natural analog to the usual $C^\infty$-functions which do
not exist in our infinite-dimensional setting (see Section \ref{Sec_Fur} for details).
We also mention that all the results of this note apply to Furstenberg transformations
on Cartesian products $(\T^n)^d$ of finite-dimensional tori $\T^n$ of any dimension
$n\ge1$. One just has to consider the particular case where the functions defining the
Furstenberg transformations on $(\T^\infty)^d$ depend only on a finite number of
variables.

Here is a brief description of the content of the note. First, we recall in Section
\ref{Sec_commutator} the needed facts on commutators of unitary operators and
regularity classes associated to them. Then, we define in Section \ref{Sec_Fur} the
Furstenberg transformations on $(\T^\infty)^d$, and prove Theorems \ref{thm_erg} and
\ref{thm_Lebesgue} on the unique ergodicity, strong mixing property and countable
Lebesgue spectrum of these transformations.

We refer to \cite{Gre11,Ji86,OP06,Rei13} and references therein for other works
building on Furstenberg transformations.\\

\noindent
{\bf Acknowledgements.} The authors thank M. M{\u{a}}ntoiu and J. Rivera-Letelier for
discussions which partly motivated this note.

%--------------------------------------------------------------------------------------
\section{Commutator methods for unitary operators}\label{Sec_commutator}
\setcounter{equation}{0}
%--------------------------------------------------------------------------------------

We recall in this section some facts borrowed from \cite{FRT13,Tie14_2} on commutator
methods for unitary operators and regularity classes associated to them.

Let $\H$ be a Hilbert space with scalar product
$\langle\;\!\cdot\;\!,\;\!\cdot\;\!\rangle$ antilinear in the first argument, denote
by $\B(\H)$ the set of bounded linear operators on $\H$, and write
$\|\;\!\cdot\;\!\|$ both for the norm on $\H$ and the norm on $\B(\H)$. Let $A$ be a
self-adjoint operator in $\H$ with domain $\dom(A)$, and take $S\in\B(\H)$. For any
$k\in\N$, we say that $S$ belongs to $C^k(A)$, with notation $S\in C^k(A)$, if the
map
\begin{equation}\label{eq_group}
\R\ni t\mapsto\e^{-itA}S\e^{itA}\in\B(\H)
\end{equation}
is strongly of class $C^k$. In the case $k=1$, one has $S\in C^1(A)$ if and only if
the quadratic form
$$
\dom(A)\ni\varphi\mapsto\big\langle\varphi,SA\;\!\varphi\big\rangle
-\big\langle A\;\!\varphi,S\varphi\big\rangle\in\C
$$
is continuous for the topology induced by $\H$ on $\dom(A)$. We denote by $[S,A]$
the bounded operator associated with the continuous extension of this form, or
equivalently $-i$ times the strong derivative of the function \eqref{eq_group} at $t=0$.

A condition slightly stronger than the inclusion $S\in C^1(A)$ is provided by the
following definition\hspace{1pt}: $S$ belongs to $C^{1+0}(A)$, with notation
$S\in C^{1+0}(A)$, if $S\in C^1(A)$ and if $[A,S]$ satisfies the Dini-type condition
$$
\int_0^1\frac{\d t}t\;\!\big\|\e^{-itA}[A,S]\e^{itA}-[A,S]\big\|<\infty.
$$
As banachisable topological vector spaces, the sets $C^2(A)$, $C^{1+0}(A)$, $C^1(A)$
and $C^0(A)\equiv\B(\H)$ satisfy the continuous inclusions \cite[Sec.~5.2.4]{ABG96}
$$
C^2(A)\subset C^{1+0}(A)\subset C^1(A)\subset C^0(A).
$$

Now, let $U\in C^1(A)$ be a unitary operator with (complex) spectral measure
$E^U(\;\!\cdot\;\!)$ and spectrum $\sigma(U)\subset\mathbb S^1:=\{z\in\C\mid|z|=1\}$.
If there exist a Borel set $\Theta\subset\mathbb S^1$, a number $a>0$ and a compact
operator $K\in\B(\H)$ such that
\begin{equation}\label{Mourre_U}
E^U(\Theta)\;\!U^*[A,U]\;\!E^U(\Theta)\ge a\;\!E^U(\Theta)+K,
\end{equation}
then one says that $U$ satisfies a Mourre estimate on $\Theta$ and that $A$ is a
conjugate operator for $U$ on $\Theta$. Also, one says that $U$ satisfies a strict
Mourre estimate on $\Theta$ if \eqref{Mourre_U} holds with $K=0$. One of the
consequences of a Mourre estimate is to imply spectral properties for $U$ on $\Theta$.
We recall here these spectral results in the case $U\in C^{1+0}(A)$. We also recall
a result on the strong mixing property of $U$ in the case $U\in C^1(A)$
(see \cite[Thm.~2.7 \& Rem.~2.8]{FRT13} and \cite[Thm.~3.1]{Tie14_2} for more details).

\begin{Theorem}[Absolutely continuous spectrum]\label{thm_abs}
Let $U$ and $A$ be respectively a unitary and a self-ajoint operator in a Hilbert
space $\H$, with $U\in C^{1+0}(A)$. Suppose there exist an open set
$\Theta\subset\mathbb S^1$, a number $a>0$ and a compact operator $K\in\B(\H)$ such
that
\begin{equation}\label{Mourre_U_ter}
E^U(\Theta)\;\!U^*[A,U]\;\!E^U(\Theta)\ge a\;\!E^U(\Theta)+K.
\end{equation}
Then, $U$ has at most finitely many eigenvalues in $\Theta$, each one of finite
multiplicity, and $U$ has no singular continuous spectrum in $\Theta$. Furthermore,
if \eqref{Mourre_U_ter} holds with $K=0$, then $U$ has only purely absolutely
continuous spectrum in $\Theta$ (no singular spectrum).
\end{Theorem}

\begin{Theorem}[Strong mixing]\label{thm_strong}
Let $U$ and $A$ be respectively a unitary and a self-adjoint operator in a Hilbert
space $\H$, with $U\in C^1(A)$. Assume that the strong limit
$$
D:=\slim_{N\to\infty}\frac1N\sum_{\ell=0}^{N-1}U^\ell\big([A,U]U^{-1}\big)U^{-\ell}
$$
exists, and suppose that $\eta(D)\;\!\dom(A)\subset\dom(A)$ for each
$\eta\in C^\infty_{\rm c}(\R\setminus\{0\})$. Then,
\begin{enumerate}
\item[(a)] $\lim_{N\to\infty}\big\langle\varphi,U^N\psi\big\rangle=0$ for each
$\varphi\in\ker(D)^\perp$ and $\psi\in\H$,
\item[(b)] $U|_{\ker(D)^\perp}$ has purely continuous spectrum.
\end{enumerate}
\end{Theorem}

%--------------------------------------------------------------------------------------
\section{Furstenberg transformations on Cartesian products of infinite-dimensional
tori}\label{Sec_Fur}
\setcounter{equation}{0}
%--------------------------------------------------------------------------------------

Let $\T^\infty\simeq\R^\infty/\Z^\infty$ be the infinite-dimensional torus with
elements $x\equiv\{x_k\}_{k=1}^\infty$, and let
$\widehat{\T^\infty}$ be the dual group of $\T^\infty$. For each
$n\in\N_{\ge1}$, let $\mu_n$ be the normalised Haar measure on $(\T^\infty)^n$, and
let $\H_n:=\ltwo\big((\T^\infty)^n,\mu_n\big)$ be the corresponding Hilbert space.
In analogy to Furstenberg transformations on Cartesian products of one-dimensional tori
(see \cite[Sec.~2.3]{Fur61}), we consider Furstenberg transformations
on Cartesian products of infinite-dimensional tori $T_d:(\T^\infty)^d\to(\T^\infty)^d$, $d\in\N_{\ge2}$, given by
$$
T_d(x_1,x_2,\ldots,x_d)
:=\big(x_1+\alpha,x_2+\phi_1(x_1),\ldots,x_d+\phi_{d-1}(x_1,x_2,\ldots,x_{d-1})\big),
$$
with $\alpha\in\T^\infty$ such that $\{n\hspace{1pt}\alpha\}_{n\in\Z}$ is dense in $\T^\infty$ and
$\phi_j\in C\big((\T^\infty)^j;\T^\infty\big)$ for each $j\in\{1,\ldots,d-1$\}.
Since $T_d$ is invertible and preserves the measure $\mu_d$, the Koopman operator
$$
W_d:\H_d\to\H_d,\quad\varphi\mapsto\varphi\circ T_d,
$$
is a unitary operator. Furthermore, $W_d$ is reduced by the orthogonal decompositions
$$
\H_d
=\H_1\bigoplus_{j\in\{2,\ldots,d\}}\big(\H_j\ominus\H_{j-1}\big)
=\H_1\bigoplus_{j\in\{2,\ldots,d\},\,\chi\in\widehat{\T^\infty}\setminus\{1\}}
\H_{j,\chi},
\quad\H_{j,\chi}:=\big\{\eta\otimes\chi\mid\eta\in\H_{j-1}\big\},
$$
and the restriction $W_d|_{\H_{j,\chi}}$ is unitarily equivalent to the unitary
operator
$$
U_{j,\chi}\;\!\eta
:=\big(\chi\circ\phi_{j-1}\big)W_{j-1}\eta,\quad\eta\in\H_{j-1}.
$$

In order to define later a conjugate operator for $U_{j,\chi}$,
we first define a suitable group of translations
on $(\T^\infty)^{j-1}$. For this, we choose $\{y_t\}_{t\in\R}$ an ergodic continuous
one-parameter subgroup of $\T^\infty$ such that each map
\begin{equation}\label{ass_C1}
\R\ni t\mapsto\big((y_t)_1,\ldots,(y_t)_n\big)\in\T^n,\quad n\in\N_{\ge1},
\end{equation}
is of class $C^1$.
An example of such a subgroup $\{y_t\}_{t\in\R}$ is for instance given by the formula
$$
(y_t)_k=y^kt \mod\Z,\quad t\in\R,~k\in\N_{\ge1},
$$
with $y\in\R$ a transcendental number (see \cite[Ex.~4.1.1]{CFS82}).
Then, we associate to $\{y_t\}_{t\in\R}$ the translation flow
$$
F_{j-1,t}(x_1,\ldots,x_{j-1}):=(x_1,\ldots,x_{j-1}+y_t),
\quad t\in\R,~(x_1,\ldots,x_{j-1})\in(\T^\infty)^{j-1},
$$
and the translation operators $V_{j-1,t}:\H_{j-1}\to\H_{j-1}$ given by
$V_{j-1,t}\hspace{1pt}\eta:=\eta\circ F_{j-1,t}$.
Due to the continuity of the map $\R\ni t\mapsto y_t\in\T^\infty$ and of the group operation,
the family $\{V_{j-1,t}\}_{t\in\R}$ defines a strongly continuous unitary group in
$\H_{j-1}$ with self-adjoint generator
$$
H_{j-1}\eta:=\slim_{t\to0}it^{-1}\big(V_{j-1,t}-1\big)\eta,
\quad\eta\in\dom(H_{j-1}):=\left\{\eta\in\H_{j-1}\mid
\lim_{t\to0}|t|^{-1}\big\|\big(V_{j-1,t}-1\big)\eta\big\|<\infty\right\}.
$$
When dealing with differential operators on compact manifolds,
one typically does the calculations on an appropriate core of the
operators such as the set of $C^\infty$-functions. But here, the are
no such functions on $(\T^\infty)^{j-1}$, since $\T^\infty$ is not a
manifold ($\T^\infty$ does not admit any differentiable
structure modeled on locally convex spaces, see \cite[Sec.~10.2]{Bog97} for details).
So, we use instead the set $\mathcal B_{j-1}$ of Bruhat test functions on
$(\T^\infty)^{j-1}$, whose definition is the following (see \cite[Sec.~2.2]{BS-C03} or
\cite{Bru61} for details). Set $\T^0:=\{0\}$, and for each $n\in\N$ let
$\pi_n:(\T^\infty)^{j-1}\to(\T^n)^{j-1}$ be the projection given by
$$
\pi_n(x_1,\ldots,x_{j-1}):=
\begin{cases}
(0,\ldots,0) & \hbox{if}~~n=0\\
\big((x_1)_1,\ldots,(x_1)_n,\ldots,(x_{j-1})_1,\ldots,(x_{j-1})_n\big)
& \hbox{if}~~n\ge1.
\end{cases}
$$
Then,
$$
\mathcal B_{j-1}
:=\bigcup_{n\in\N}\big\{\zeta\circ\pi_n\mid
\zeta\in C^\infty\big((\T^n)^{j-1}\big)\big\},
$$
that is, $\mathcal B_{j-1}$ is the set of all functions on $(\T^\infty)^{j-1}$ that
are obtained by lifting to $(\T^\infty)^{j-1}$ any $C^\infty$-function on one of the
Lie groups $(\T^n)^{j-1}$.
The set $\mathcal B_{j-1}$ is dense in $\H_{j-1}$, it is left invariant by the group
$\{V_{j-1,t}\}_{t\in\R}$, and satisfies the inclusion $\mathcal B_{j-1}\subset\dom(H_{j-1})$
(to show the latter, one has to use the $C^1$-assumption \eqref{ass_C1}).
Therefore, it follows from Nelson's theorem \cite[Prop.~5.3]{Amr09} that $H_{j-1}$ is
essentially self-adjoint on $\mathcal B_{j-1}$.

Now, if the (multiplication operator valued) map 
$
\R\ni t\mapsto\chi\circ\phi_{j-1}\circ F_{j-1,t}\in\B(\H_{j-1})
$
is strongly of class $C^1$, then $\chi\circ\phi_{j-1}\in C^1(H_{j-1})$ since
$$
\chi\circ\phi_{j-1}\circ F_{j-1,t}
=V_{j-1,t}\big(\chi\circ\phi_{j-1}\big)V_{j-1,-t}
=\e^{-itH_{j-1}}\big(\chi\circ\phi_{j-1}\big)\e^{itH_{j-1}}.
$$
Thus, the operator
\begin{align*}
g_{j,\chi}
:=\big[H_{j-1},(\chi\circ\phi_{j-1}\big)\big]\big(\overline\chi\circ\phi_{j-1}\big)
&\equiv\hbox{s\;\!-\;\!}\frac\d{\d t}\;\!\big(i\;\!\chi\circ\phi_{j-1}\circ F_{j-1,t}\big)
\big(\overline\chi\circ\phi_{j-1}\big)\Big|_{t=0}\\
&\equiv\hbox{s\;\!-\;\!}\frac\d{\d t}\;\!i\;\!
\chi\circ\big(\phi_{j-1}\circ F_{j-1,t}-\phi_{j-1}\big)\Big|_{t=0}
\end{align*}
is a bounded self-adjoint multiplication operator in $\H_{j-1}$.

\begin{Lemma}\label{lemma_C1}
Take $j\in\{2,\ldots,d\}$ and $\chi\in\widehat{\T^\infty}\setminus\{1\}$, and
assume that the map $\R\ni t\mapsto\chi\circ\phi_{j-1}\circ F_{j-1,t}\in\B(\H_{j-1})$
is strongly of class $C^1$. Then, $U_{j,\chi}\in C^1(H_{j-1})$ with
$[H_{j-1},U_{j,\chi}]=g_{j,\chi}U_{j,\chi}$.
\end{Lemma}

\begin{proof}
Take $\eta\in\B_{j-1}$. Then, the commutation of $V_{j-1,t}$ and $W_{j-1}$
and the differentiability assumption imply that
\begin{align*}
&\big\langle H_{j-1}\eta,U_{j,\chi}\eta\big\rangle_{\H_{j-1}}
-\big\langle\eta,U_{j,\chi}H_{j-1}\eta\big\rangle_{\H_{j-1}}\\
&=i\;\!\frac\d{\d t}\left\{-\big\langle V_{j-1,t}\eta,
\big(\chi\circ\phi_{j-1}\big)\;\!W_{j-1}\eta\big\rangle_{\H_{j-1}}
-\big\langle\eta,\big(\chi\circ\phi_{j-1}\big)W_{j-1}V_{j-1,t}\eta
\big\rangle_{\H_{j-1}}\right\}\Big|_{t=0}\\
&=i\;\!\frac\d{\d t}\;\!\big\langle\eta,
\big(\chi\circ\phi_{j-1}\circ F_{j-1,t}-\chi\circ\phi_{j-1}\big)
W_{j-1}V_{j-1,t}\eta\big\rangle_{\H_{j-1}}\Big|_{t=0}\\
&=\frac\d{\d t}\;\!\big\langle\eta,
\big\{\big(i\;\!\chi\circ\phi_{j-1}\circ F_{j-1,t}\big)
\big(\overline\chi\circ\phi_{j-1}\big)-1\big\}
\;\!U_{j,\chi}V_{j-1,t}\eta\big\rangle_{\H_{j-1}}\Big|_{t=0}\\
&=\frac\d{\d t}\;\!\big\langle\eta,
\big\{i\;\!\chi\circ\big(\phi_{j-1}\circ F_{j-1,t}-\phi_{j-1}\big)-1\big\}
\;\!U_{j,\chi}V_{j-1,t}\eta\big\rangle_{\H_{j-1}}\Big|_{t=0}\\
&=\big\langle\eta,g_{j,\chi}\hspace{1pt}U_{j,\chi}\eta\big\rangle_{\H_{j-1}},
\end{align*}
and thus the claim follows from the boundedness of $g_{j,\chi}$ and the density of
$\B_{j-1}$ in $\dom(\H_{j-1})$.
\end{proof}

\begin{Assumption}\label{ass_phi}
For each $j\in\{2,\ldots,d\}$, the map $\phi_{j-1}\in C\big((\T^\infty)^{j-1};\T^\infty\big)$
satisfies $\phi_{j-1}=\xi_{j-1}+\eta_{j-1}$, where
\begin{enumerate}
\item[(i)] $\xi_{j-1}\in C\big((\T^\infty)^{j-1};\T^\infty\big)$ is a group
homomorphism such that
$$
\T^\infty\ni x_{j-1}\mapsto(\chi\circ\xi_{j-1})(0,\ldots,0+x_{j-1})\in\S^1
$$
is nontrivial for each $\chi\in\widehat{\T^\infty}\setminus\{1\}$,
\item[(ii)] $\eta_{j-1}\in C\big((\T^\infty)^{j-1};\T^\infty\big)$ is such that there exists
$\widetilde\eta_{j-1}\in C\big((\T^\infty)^{j-1};\R^\infty\big)$ with
$$
\eta_{j-1}(x_1,\ldots,x_{j-1})=\big(\widetilde\eta_{j-1}(x_1,\ldots,x_{j-1})
~({\rm mod}~\Z^\infty)\big)\quad\hbox{for each }(x_1,\ldots,x_{j-1})\in(\T^\infty)^{j-1},
$$
and with $\R\ni t\mapsto\widetilde\eta_{j-1,k}\circ F_{j-1,t}\in\B(\H_{j-1})$ strongly
of class $C^1$ for each $k\in\N_{\ge1}$.
\end{enumerate}
\end{Assumption}

In the next theorem, we use the fact that the map
$$
\R\ni t\mapsto\big(\chi\circ\xi_{j-1}\big)(0,\ldots,0+y_t)\in\S^1
$$
is a character on $\R$, and thus of class $C^\infty$. We also use the notation
$$
\xi_{j-1}^{(\chi)}
:=\frac\d{\d t}\big(\chi\circ\xi_{j-1}\big)(0,\ldots,0+y_t)\Big|_{t=0}\in i\;\!\R.
$$

\begin{Theorem}[Strong mixing and unique ergodicity]\label{thm_erg}
Suppose that Assumption \ref{ass_phi} is satisfied. Then, $W_d$ is strongly mixing in
$\H_d\ominus\H_1$ and $T_d$ is uniquely ergodic with respect to $\mu_d$.
\end{Theorem}

These results of strong mixing and unique ergodicity are an extension to the case of
infinite-dimensional tori of results previously obtained in the case of
one-dimensional tori (see \cite[Thm.~2.1]{Fur61} for the unique ergodicity and
\cite[Rem.~1]{ILR93} for the strong mixing property). For instance, if the functions
$\eta_{j-1}$ in Assumption \ref{ass_phi} were to depend only on a finite number of
variables, then the strong $C^1$ regularity condition on $\eta_{j-1}$ in Assumption
\ref{ass_phi}(ii) would reduce to a uniform Lipschitz condition of $\eta_{j-1}$ along
the flow $\{F_{j-1,t}\}_{t\in\R}$, as in the one-dimensional case treated by
Furstenberg in \cite[Thm.~2.1]{Fur61}.

\begin{proof}
(i) Take $j\in\{2,\ldots,d\}$, $\chi\in\widehat{\T^\infty}\setminus\{1\}$ and $t\in\R$.
Then, there exist $k_\chi\in\N_{\ge1}$ and $n_1,\ldots,n_{k_\chi}\in\Z$ such that
$$
\chi(x_{j-1})=\e^{2\pi i\sum_{k=1}^{k_\chi}n_kx_{j-1,k}},
$$
with $x_{j-1}\in\T^\infty$ and $x_{j-1,k}\in[0,1)$ the $k$-th cyclic component
of $x_{j-1}$. Therefore, we infer from Assumption \ref{ass_phi} that
\begin{align*}
\chi\circ\big(\phi_{j-1}\circ F_{j-1,t}-\phi_{j-1}\big)
&=\chi\circ\big(\xi_{j-1}\circ F_{j-1,t}-\xi_{j-1}\big)
\cdot\chi\circ\big(\eta_{j-1}\circ F_{j-1,t}-\eta_{j-1}\big)\\
&=\big(\chi\circ\xi_{j-1}\big)(0,\ldots,0+y_t)
\cdot\e^{2\pi i\sum_{k=1}^{k_\chi}n_k(\widetilde\eta_{j-1}\circ F_{j-1,t}
-\widetilde\eta_{j-1})_k},
\end{align*}
and Lemma \ref{lemma_C1} implies that $U_{j,\chi}\in C^1(H_{j-1})$ with
$[H_{j-1},U_{j,\chi}]=g_{j,\chi}\hspace{1pt}U_{j,\chi}$ and
\begin{align}
g_{j,\chi}
=\hbox{s\;\!-\;\!}\frac\d{\d t}\;\!i\;\!
\chi\circ(\phi_{j-1}\circ F_{j-1,t}-\phi_{j-1})\Big|_{t=0}
&=i\;\!\xi_{j-1}^{(\chi)}-2\pi\sum_{k=1}^{k_\chi}n_k
\left(\hbox{s\;\!-\;\!}\frac\d{\d t}\;\!\widetilde\eta_{j-1,k}\circ F_{j-1,t}
\Big|_{t=0}\right)\nonumber\\
&=i\;\!\xi_{j-1}^{(\chi)}
+2\pi i\sum_{k=1}^{k_\chi}n_k\;\!(H_{j-1}\widetilde\eta_{j-1,k}).\label{eq_g}
\end{align}

(ii) We now proceed by induction on $d$ to prove the claims. For the case $d=2$, we take $\chi\in\widehat{\T^\infty}\setminus\{1\}$ and note that point (i) implies
\begin{align*}
D_{2,\chi}
&:=\slim_{N\to\infty}\frac1N\sum_{\ell=0}^{N-1}(U_{2,\chi})^\ell
\big([H_1,U_{2,\chi}](U_{2,\chi})^{-1}\big)(U_{2,\chi})^{-\ell}\\
&=\slim_{N\to\infty}\frac1N\sum_{\ell=0}^{N-1}g_{2,\chi}\circ(T_1)^\ell\\
&=i\;\!\xi_1^{(\chi)}+2\pi i\sum_{k=1}^{k_\chi}n_k
\left(\slim_{N\to\infty}\frac1N\sum_{\ell=0}^{N-1}(H_1\widetilde\eta_{1,k})
\circ(T_1)^\ell\right).
\end{align*}
Since $T_1$ is ergodic and $H_1\widetilde\eta_{1,k}\in\linf(\T^\infty)$, it follows by
Birkhoff's pointwise ergodic theorem and Lebesgue dominated convergence theorem that
$$
D_{2,\chi}
=i\;\!\xi_1^{(\chi)}+2\pi i\sum_{k=1}^{k_\chi}n_k
\int_{\T^\infty}\d\mu_1\;\!(H_1\widetilde\eta_{1,k})
=i\;\!\xi_1^{(\chi)}+2\pi i\sum_{k=1}^{k_\chi}n_k
\int_{\T^\infty}\d\mu_1\;\!\big\langle1,H_1\widetilde\eta_{1,k}\big\rangle_{\H_1}
=i\;\!\xi_1^{(\chi)}.
$$
Now, since the character
$\T^\infty\ni x_{j-1}\mapsto(\chi\circ\xi_1)(0,\ldots,0+x_{j-1})\in\S^1$
is nontrivial and the subgroup $\{y_t\}_{t\in\R}$ ergodic, we have
$\xi_1^{(\chi)}\ne0$ (see \cite[Thm.~4.1.1']{CFS82}). Thus, $D_{2,\chi}\ne0$, and we deduce from
Theorem \ref{thm_strong}(a) that $U_{2,\chi}$ is strongly mixing. Since this is true for each
$\chi\in\widehat{\T^\infty}\setminus\{1\}$, and since $U_{2,\chi}$ is unitarily
equivalent to $W_2|_{\H_{2,\chi}}$, we infer that $W_2$ is strongly mixing in $\oplus_{\chi\in\widehat{\T^\infty}\setminus\{1\}}\H_{2,\chi}=\H_2\ominus\H_1$.

To show that $T_2$ is uniquely ergodic with respect to $\mu_2$, we take an eigenvector of
$W_2$ with eigenvalue $1$, that is, a vector $\varphi\in\H_2$ such that $W_2\;\!\varphi=\varphi$.
Since $W_2$ is strongly mixing in $\H_2\ominus\H_1$, $W_2$ has purely continuous spectrum
in $\H_2\ominus\H_1$ (see Theorem \ref{thm_strong}(b)), and thus
$\varphi=\eta\otimes1$ for some $\eta\in\H_1$. So,
$$
W_2\;\!\varphi=\varphi
\iff W_2\;\!(\eta\otimes1)=\eta\otimes1
\iff W_1\;\!\eta=\eta,
$$
and thus $\eta$ is an eigenvector of $W_1$ with eigenvalue $1$. It follows that
$\eta$ is constant
$\mu_1$-almost everywhere due to the ergodicity of $T_1$. Therefore, $\varphi$ is constant
$\mu_2$-almost everywhere, and $T_2$ is ergodic. This implies that $T_2$ is uniquely ergodic
because ergodicity implies unique ergodicity for
transformations such as $T_2$ (see \cite[Thm.~4.2.1]{CFS82}).

Now, assume the claims are true for $d-1\ge1$.
Then, $W_{d-1}$ is strongly mixing in $\H_{d-1}\ominus\H_1$. Furthermore, a calculation as in the
case $d=2$ shows that $W_d$ is strongly mixing in
$
\oplus_{\chi\in\widehat{\T^\infty}\setminus\{1\}}\H_{d,\chi}
=\H_d\ominus\H_{d-1}
$.
This implies that $W_d$ is strongly mixing in
$$
\big(\H_{d-1}\ominus\H_1\big)\oplus\big(\H_d\ominus\H_{d-1}\big)=\H_d\ominus\H_1.
$$
This, together with the unique ergodicity of $T_{d-1}$, allows us to show that $T_d$
is uniquely ergodic as in the case $d=2$.
\end{proof}

We know from the proof of Theorem \ref{thm_erg} that if Assumption \ref{ass_phi} is satisfied,
then $U_{j,\chi}\in C^1(H_{j-1})$ with $[H_{j-1},U_{j,\chi}]=g_{j,\chi}\hspace{1pt}U_{j,\chi}$.
So, it follows from \cite[Sec.~4]{FRT13} that the operator
$$
A_{j,\chi}^{(N)}\eta
:=\frac1N\sum_{\ell=0}^{N-1}(U_{j,\chi})^\ell\;\!H_{j-1}(U_{j,\chi})^{-\ell}\;\!\eta,
\quad N\in\N_{\ge1},~\eta\in\dom\big(A_{j,\chi}^{(N)}\big):=\dom(H_{j-1}),
$$
is self-adjoint, and that $U_{j,\chi}\in C^1\big(A_{j,\chi}^{(N)}\big)$
with
$$
\big[A_{j,\chi}^{(N)},U_{j,\chi}\big]=g_{j,\chi}^{(N)}U_{j,\chi}
\qquad\hbox{and}\qquad
g_{j,\chi}^{(N)}:=\frac1N\sum_{\ell=0}^{N-1}g_{j,\chi}\circ(T_{j-1})^\ell.
$$

\begin{Theorem}[Countable Lebesgue spectrum]\label{thm_Lebesgue}
Suppose that Assumption \ref{ass_phi} is satisfied, and assume for each $j\in\{2,\ldots,d\}$
and $k\in\N_{\ge1}$ that $H_{j-1}\widetilde\eta_{j-1,k}\in C\big((\T^\infty)^{j-1}\big)$ and that
\begin{equation}\label{ass_C0}
\int_0^1\frac{\d t}t\;\!\big\|(H_{j-1}\widetilde\eta_{j-1,k})\circ F_{j-1,t}
-(H_{j-1}\widetilde\eta_{j-1,k})\big\|_{\linf((\T^\infty)^{j-1})}<\infty.
\end{equation}
Then, $W_d$ has countable Lebesgue spectrum in $\H_d\ominus\H_1$.
\end{Theorem}

\begin{proof}
Take $j\in\{2,\ldots,d\}$ and $\chi\in\widehat{\T^\infty}\setminus\{1\}$. Then, we
know from Lemma \ref{lemma_C1} that $U_{j,\chi}\in C^1(H_{j-1})$ with
$[H_{j-1},U_{j,\chi}]=g_{j,\chi}\hspace{1pt}U_{j,\chi}$. Furthermore, \eqref{eq_g} and
\eqref{ass_C0} imply that
\begin{align*}
&\int_0^1\frac{\d t}t\;\!\big\|\e^{-itH_{j-1}}g_{j,\chi}\e^{itH_{j-1}}-g_{j,\chi}\big\|
_{\B(\H_{j-1})}\\
&=\int_0^1\frac{\d t}t\;\!\big\|g_{j,\chi}\circ F_{j-1,t}-g_{j,\chi}\big\|
_{\linf((\T^\infty)^{j-1})}\\
&\le2\pi\sum_{k=1}^{k_\chi}|n_k|\int_0^1\frac{\d t}t\;\!
\big\|(H_{j-1}\widetilde\eta_{j-1,k})\circ F_{j-1,t}
-(H_{j-1}\widetilde\eta_{j-1,k})\big\|_{\linf((\T^\infty)^{j-1})}\\
&<\infty.
\end{align*}
So, we obtain that $U_{j,\chi}\in C^{1+0}(H_{j-1})$ with
$[H_{j-1},U_{j,\chi}]=g_{j,\chi}\hspace{1pt}U_{j,\chi}$, and thus deduce from \cite[Sec.~4]{FRT13}
that $U_{j,\chi}\in C^{1+0}\big(A_{j,\chi}^{(N)}\big)$ with
$\big[A_{j,\chi}^{(N)},U_{j,\chi}\big]=g_{j,\chi}^{(N)}U_{j,\chi}$.
Now, $H_{j-1}\widetilde\eta_{j-1,k}\in C\big((\T^\infty)^{j-1}\big)$ and
$T_{j-1}$ is uniquely ergodic with respect to $\mu_{j-1}$ due to Theorem \ref{thm_erg}. Thus,
we infer from \eqref{eq_g} that
\begin{align*}
\lim_{N\to\infty}g_{j,\chi}^{(N)}
&=i\;\!\xi_{j-1}^{(\chi)}+2\pi i\sum_{k=1}^{k_\chi}n_k\;\!\frac1N\sum_{\ell=0}^{N-1}
(H_{j-1}\widetilde\eta_{j-1,k})\circ(T_{j-1})^\ell\\
&=i\;\!\xi_{j-1}^{(\chi)}+2\pi i\sum_{k=1}^{k_\chi}n_k\int_{(\T^\infty)^{j-1}}\d\mu_{j-1}\;\!
\big\langle1,H_{j-1}\widetilde\eta_{j-1,k}\big\rangle_{\H_{j-1}}\\
&=i\;\!\xi_{j-1}^{(\chi)}
\end{align*}
uniformly on $(\T^\infty)^{j-1}$. Since $\xi_{j-1}^{(\chi)}\ne0$ by the proof of
Theorem \ref{thm_erg}, one has $\big|g_{j,\chi}^{(N)}\big|>0$ if $N$ is large enough.
So, $\big|g_{j,\chi}^{(N)}\big|\ge a$ with
$a:=\inf_{x\in(\T^\infty)^{j-1}}\big|g_{j,\chi}^{(N)}(x)\big|>0$, and $U_{j,\chi}$ satisfies the following strict Mourre estimate on $\S^1$\hspace{1pt}:
$$
(U_{j,\chi})^*\big[{\rm sgn}\big(g_{j,\chi}^{(N)}\big)A_{j,\chi}^{(N)},U_{j,\chi}\big]
=(U_{j,\chi})^*\big|g_{j,\chi}^{(N)}\big|\;\!U_{j,\chi}
\ge a.
$$
Therefore, all the assumptions of Theorem \ref{thm_abs} are satisfied, and thus
$U_{j,\chi}$ has purely absolutely continuous spectrum. Since this occurs for each
$j\in\{2,\ldots,d\}$ and $\chi\in\widehat{\T^\infty}\setminus\{1\}$, and since
$U_{j,\chi}$ is unitarily equivalent to $W_d|_{\H_{j,\chi}}$, one infers that $W_d$
has purely absolutely continuous spectrum in
$
\bigoplus_{j\in\{2,\ldots,d\},\,\chi\in\widehat{\T^\infty}\setminus\{1\}}\H_{j,\chi}
=\H_d\ominus\H_1
$.
Finally, the fact that $W_d$ has has countable Lebesgue spectrum in $\H_d\ominus\H_1$
can be shown in a similar way as in \cite[Lemmas~3~\&~4]{ILR93}.
\end{proof}

The result of Theorem \ref{thm_Lebesgue} is an extension to the case of
infinite-dimensional tori of results previously obtained in the case of
one-dimensional tori (see \cite[Cor.~3]{ILR93} or \cite[Thm.~5.3]{Tie13}).

%--------------------------------------------------------------------------------------
%\bibliography{../bibliographie/bibliographie}
%--------------------------------------------------------------------------------------

\def\polhk#1{\setbox0=\hbox{#1}{\ooalign{\hidewidth
  \lower1.5ex\hbox{`}\hidewidth\crcr\unhbox0}}}
  \def\polhk#1{\setbox0=\hbox{#1}{\ooalign{\hidewidth
  \lower1.5ex\hbox{`}\hidewidth\crcr\unhbox0}}}
  \def\polhk#1{\setbox0=\hbox{#1}{\ooalign{\hidewidth
  \lower1.5ex\hbox{`}\hidewidth\crcr\unhbox0}}} \def\cprime{$'$}
  \def\cprime{$'$} \def\polhk#1{\setbox0=\hbox{#1}{\ooalign{\hidewidth
  \lower1.5ex\hbox{`}\hidewidth\crcr\unhbox0}}}
  \def\polhk#1{\setbox0=\hbox{#1}{\ooalign{\hidewidth
  \lower1.5ex\hbox{`}\hidewidth\crcr\unhbox0}}} \def\cprime{$'$}
  \def\cprime{$'$}

%--------------------------------------------------------------------------------------

\end{document}